\documentclass[11pt]{amsart}
\usepackage{amsmath}
\usepackage{amsthm}
\usepackage{amsfonts}
\usepackage{amssymb}
\usepackage{graphicx}
\usepackage{tikz}
\usepackage{verbatim}
\usepackage{amscd}
\usepackage{framed}

\usepackage{tikz}
\usepackage{tikz-cd}
\usetikzlibrary{matrix, calc, arrows}

\usepackage{pgfplots}
\pgfplotsset{every axis/.append style={
                    axis x line=middle,    
                    axis y line=middle,    
                    axis line style={<->}, 
                    xlabel={$x$},          
                    ylabel={$y$},          
                    }}
\tikzset{>=stealth}

\graphicspath{ {images/} }
\newtheorem{Theorem}{Theorem}[section]
\newtheorem{Proposition}[Theorem]{Proposition}
\newtheorem{Lemma}[Theorem]{Lemma}

\newtheorem{Hypothesis}[Theorem]{Hypothesis}
\theoremstyle{definition}
\newtheorem{Definition}{Definition}[section]
\theoremstyle{remark}

\numberwithin{equation}{section}

\newcommand{\R}{{\mathbb R}}
\newcommand{\C}{{\mathbb C}}

\newcommand{\B}{{\mathcal B}}
\newcommand{\T}{{\mathcal T}}

\newcommand{\M}{{\mathcal M}}
\newcommand{\HH}{{\mathcal H}}
\newcommand{\J}{{\mathcal J}}

\begin{document}
\title[Restrictions on the existence of a canonical system flow hierarchy]
{Restrictions on the existence of a canonical system flow hierarchy}

\author{Injo Hur and Darren C. Ong}

\address{Department of Mathematics Education, Chonnam National University,
77 Yongbong-ro, Buk-gu, Gwangju 61186, Republic of Korea}
\email{injohur@jnu.ac.kr}

\address{Department of Mathematics,  Xiamen University Malaysia, Jalan Sunsuria, Bandar Sunsuria, 43900 Sepang, Selangor, Malaysia}
\email{darrenong@xmu.edu.my}
\urladdr{https://www.drong.my}


\date{\today}




\begin{abstract}
The KdV hierarchy is a family of evolutions on a Schr\"odinger operator that preserves its spectrum. Canonical systems are a generalization of Schr\"odinger operators, that nevertheless share many features with Schr\"odinger operators. Since this is a very natural generalization, one would expect that it would also be straightforward to build a hierarchy of isospectral evolutions on canonical systems analogous to the KdV hierarchy. Surprisingly, we show that there are many obstructions to constructing a hierarchy of flows on canonical systems that obeys the standard assumptions of the KdV hierarchy. This suggests that we need a more sophisticated approach to develop such a hierarchy, if it is indeed possible to do so.\end{abstract}
\maketitle

\section{Introduction}
The Korteweg-de Vries (or KdV) equation 
\begin{equation}\label{KdV equation}
\dfrac{\partial}{\partial  t} V( x, t)=-\frac14 \dfrac{\partial^3}{\partial x^3} V( x, t)+\frac32 V( x, t) \dfrac{\partial}{\partial x}V(x,t)
\end{equation}
is a well-known mathematical model for waves on the surface of shallow water. This KdV equation has numerous applications in mathematics and physics. For example, if $V(x,t)$ is considered as a potential function of a Schr\"odinger operator 
\begin{equation}\label{Schrodinger}
L(x,t)=-\frac{d^2}{dx^2} +V(x,t),
\end{equation}
evolving this potential via the KdV equation (thought of as an evolution in time $t$) will leave the spectrum of the Schr\"odinger operator invariant. In other words, the KdV equation is associated to an \emph{isospectral} evolution of the Schr\"odinger operator. 

This concept may extend to a whole family of smooth evolutions of $V$, known as the KdV hierarchy, all of which preserve the spectrum of the Schr\"odinger operator containing $V$ as a potential. The KdV hierarchy is of crucial importance in the inverse spectral theory of Schr\"odinger operators, and the KdV equation \eqref{KdV equation} is the simplest nontrivial member of this hierarchy. Please see \cite{Gesztesy-Holden} for a detailed explanation of the KdV hierarchy.

In this paper, we would like to see what happens if we apply ideas from the KdV hierarchy to  \emph{canonical systems}. Canonical systems can be thought of as a generalization of eigenvalue equations of Schr\"odinger operators. See \cite{RemBook} for a thorough discussion of the spectral theory of canonical systems. This generalization is best understood from the perspective of Weyl-Titchmarsh $m$-functions which are thought of as spectral data. These are Herglotz functions, that is, functions that holomorphically map the upper half plane to itself. The spectrum of a Schr\"odinger operator on the non-negative half-line $[0,\infty)$ can be derived from the limiting behavior of its corresponding $m$-function (or its corresponding pair of $m$-functions, if we consider a Schr\"odinger operator on the whole line $\mathbb R$ instead). Each Schr\"odinger operator on $[0,\infty)$ corresponds to a Herglotz function in this way, whereas a Herglotz function might not be associated with a Schr\"odinger operator. On the other hand, every Herglotz function corresponds to a canonical system on $[0,\infty)$, and it is in this sense that canonical systems are the ``highest generalization" of the Schr\"odinger spectral problem.  

 Since the generalization from Schr\"odinger operators to canonical systems is so natural in spectral theory, naively, one would expect that we can find a family of isospectral canonical system flows that behave like KdV flows. Indeed, many of the tools that we use for the spectral analysis of Schr\"odinger operators and the development of the KdV hierarchy are also present in the canonical system setting. For example, canonical systems also have a transfer matrix formalism.

Surprisingly, in our paper we show that it is very difficult to construct a family of isospectral flows on canonical systems that is similar to the KdV hierarchy. More precisely, as a consequence of the zero-curvature equation \eqref{zceqnforcs} for canonical systems we derive some conditions for the existence of a hierarchy of flows. We then demonstrate that there are difficult obstructions to satisfying these conditions in the canonical system setting, which are not present in the Schr\"odinger operator setting even though KdV flows also have to satisfy their pertinent  zero-curvature equation.

We are not asserting that non-trivial isospectral flows on canonical systems would be impossible to find. The best way to clarify our result is as follows: in a very natural way, each member of the KdV hierarchy corresponds to a Hamiltonian on a canonical system. This Hamiltonian is a $2\times 2$ matrix whose entries are polynomials with real coefficients. Our result demonstrates that for canonical systems, it is impossible to create isospectral flows corresponding to polynomials of degree two or higher, unless the coefficients of the canonical system obey a restrictive technical condition at every time-step of the evolution. This restriction is not present at all in the KdV setting. 

The main takeaway from this paper is that a direct approach to constructing an isospectral flow hierarchy for canonical systems is probably not viable. It is, however, possible that a more sophisticated approach might be fruitful, perhaps by incorporating the ``twisted shift" idea introduced in \cite{Remflow}.

This paper is organized as follows.  In Section \ref{Section:KdVreview} the KdV hierarchy is reviewed via its zero-curvature equation in order to examine the tools to achieve the KdV hierarchy. The similar tools will be employed on canonical system flows and then the related zero-curvature equation is obtained in Section \ref{Section3}. It is in the last section that we realize the condition to canonical system flows similar to isospectral KdV flows, and we then see why these isospectral canonical system flows are hard to construct.\\

\textit{Acknowledgement:} 
The authors would like to deeply thank Christian Remling for valuable discussions. The second author would also like to thank Jessica Liang Yei Shan for proofreading help. The first author  was supported by Basic Science Research Program through the National Research Foundation of Korea(NRF) funded by the Ministry of Education (NRF-2016R1D1A1B03931764 and NRF-2019R1F1A1061300), and the second author was supported by a grant from the Fundamental Research Grant Scheme from the Malaysian Ministry of Education (Grant No: \\ FRGS/1/2018/STG06/XMU/02/1) and two Xiamen University Malaysia Research Funds (Grant Nos: XMUMRF/2018-C1/IMAT/0001 and \\
XMUMRF/2020-C5/IMAT/0011).

\section{Review for the KdV hierarchy}\label{Section:KdVreview}
Let us review KdV flows and their hierarchy. There are two equivalent ways to construct the KdV hierarchy: the Lax pair formalism and the zero-curvature equation. In this section we exclusively employ the latter approach, since it will be easier to generalize to canonical systems later. For related work, see \cite{Gesztesy-Holden} for an exposition of the KdV hierarchy by means of the Lax pair formalism and \cite{Ong,D&R} for a treatment of Toda flows (a discrete analogue of KdV flows) using a zero-curvature equation.

We define Schr\"odinger operators $L$ as in \eqref{Schrodinger}, where the potentials $V$ are assumed to be  real-valued functions of  two variables $x$ and $t$ such that they are $C^{\infty}$-functions in $x$ (in $\R$) and $C^1$-functions in $t$. 
We further assume that $V(x)$ is smooth in $x$, which guarantees the smoothness of the solutions $u$ to $Lu=zu$. 
Inspired by the formulation of  \cite{Remflow} we express the $t$-dependence of the Schr\"odinger operator, in the form of a $\mathbb R$-group action which acts as a shift on $L$. That is, instead of $L(x,t)$ we write $t\cdot L(x)$, and also instead of $L(x,0)$ we write $L(x)$. The advantages of this approach are expounded in \cite{Remflow}. In brief, expressing the time evolution as a $\mathbb R$-group action clarifies the role of the transfer matrix cocycle property corresponding to the evolution in time, and opens the door for potential generalizations that involve replacing the group $\mathbb R$ with a more general group.

We introduce transfer matrices that shift the operator from $x=0$ to $x=w$:
\begin{equation}\label{tmforschrodinger}
\T_\boxplus(w,z;L):=\begin{pmatrix} u(w,z;L) & v(w,z;L) \\ u_x(w,z;L) & v_x(w,z;L) \end{pmatrix}.
\end{equation}
Here $y(x)=u(x,z;L)$ and $y(x)=v(x,z;L)$ are the solutions to the eigenvalue equation $ L(x)y(x)=zy(x)$ of \eqref{Schrodinger}.  We also impose the following initial conditions for $w=0$:
\begin{eqnarray}\label{bcat0}
\begin{pmatrix}u(0,z;L) & v(0,z;L) \\  u_x(0,z;L) & v_x(0,z;L)   \end{pmatrix}
=\begin{pmatrix} 1 &0 \\ 0 &1 \end{pmatrix}.
\end{eqnarray}
Note that, since the Wronskian of two solutions to $L(x)y(x)=zy(x)$ is constant, \eqref{bcat0} implies that $\det \T_\boxplus=1$ for all $w$. 

The reason why the $\T_\boxplus$ matrices are called transfer matrices is that for any solutions $y$ to $L(x) y(x)=zy(x)$,
\begin{equation}\label{eq.Tboxplus}
\T_\boxplus(w,z;L) \begin{pmatrix} y(0,z;L) \\ y_x(0,z;L) \end{pmatrix} 
=\begin{pmatrix} y(w,z;L) \\ y_x(w,z;L) \end{pmatrix}.
\end{equation}
In other words, $\T$'s ``transfer'' the solutions from $x=0$ to $x=w$. We will sketch a justification for \eqref{eq.Tboxplus}: consider a solution $y(x)$ with initial conditions $y(0)=I_1, y'(0)=I_2$. Assume without loss of generality that $I_1^2+I_2^2=1$. Then we may write $y(x)=I_1 u(x)+I_2v(x)$ for all $x$. In that case, we can check that \eqref{eq.Tboxplus} holds.

We can also view this transfer property as an action of the transfer matrix on the Weyl-Titchmarsh $m$-functions associated with the Schr\"odinger equations $L(x)y=zy$. Assuming that $L$ is acting on the real line, these $m$-functions are defined as
\begin{equation*}
m_\pm(z; L)=\mp\frac{\tilde y_{\pm,x}(0,z;L)}{\tilde y_\pm(0,z;L) },
\end{equation*}
where $\tilde y_\pm(x,z;L)$ are solutions of the Schr\"odinger equation that are $\ell^2$ at $\pm\infty$. Using this perspective, the  $\T_\boxplus$ map acts on these $m$-functions as linear fractional transformations. More explicitly, linear fractional transformations act on $\mathbb C\cup\{\infty\}$ as
\begin{equation}\label{e.LFT}
\begin{pmatrix}
a&b\\
c&d
\end{pmatrix}
\odot z=\frac{az+b}{cz+d},
\end{equation}
and using this action, we can verify that 
\begin{equation}\label{eq:Tboxplusshifty/y}
\mathcal T_\boxplus(w,z,L)\odot   \frac{1}{m_\pm (z,L)}=\frac{1}{\mp\frac{\tilde y_{\pm,x}(w,z; L)}{\tilde y_\pm(w,z; L)}},
\end{equation}
 which can be thought of as $m_\pm$ shifted in the $x$-direction to the right by $w$ units. Alternatively, if we write $L^{(w)}$ to represent the operator $L$ shifted to the right by $w$ units (that is, $L^{(w)}(x)=L(x+w)$) we can think of \eqref{eq:Tboxplusshifty/y} as representing a transformation that shifts the Schr\"odinger operator in the folowing way:
 \begin{equation}\label{eq:Tboxplusshifty/y}
\mathcal T_\boxplus(w,z,L)\odot   m_\pm (z,L)=\ m_\pm (z,L^{(w)})
\end{equation}
 See \cite{Hur-density,K&L,Remflow} for more details.\\

We introduce the following formulation of the KdV hierarchy. The KdV hierarchy is a family of evolutions in time $t$ that obey a 
\begin{itemize}
\item cocycle property,
\item  commutativity with the shift in $x$ and
\item  polynomial recursion formalism.
\end{itemize}
 We will clarify these three items. First, we introduce another $\mathrm{SL}(2,\mathbb C)$ transfer matrix, $\mathcal T_\circledast$, which describes the evolution of the solutions (and thus $m$-functions) through time. That is, the time evolution of the $m$-functions are described by $\mathcal T_\circledast(t,z;L)\odot \frac{1}{m_\pm(z,L)}=\frac{1}{m_\pm(z,t\cdot L)}$. So $\mathcal T_\circledast$ transfers the $m$-functions (or equivalently, the $\ell^2$ solutions of the Schr\"odinger equation) from time $0$ to time $t$. We do not write $T_\circledast$ explicitly. Instead, we will later define it as a solution of a differential equation. 

 Let us make implicit the dependence of $\mathcal T_\circledast$ on $z$, to simplify notation. In order for our time evolution to make sense, we need $\mathcal T_\circledast$ to obey the following cocycle property:  
\begin{Definition}\label{Defn:cocycle} We say that $\T_\circledast$ is a \emph{cocycle}, if it satisfies the equality 
\begin{equation}\label{cocycle}
\T_\circledast (t_1+t_2;L)=\T_\circledast(t_1; t_2\cdot L)\T_\circledast(t_2;L).
\end{equation}
\end{Definition}

This means that updating the Schr\"odinger operator from $t=0$ to $t=t_1+t_2$ should give us the same result as updating the operator first from $t=0$ to $t=t_2$, and then from $t=t_2$ to $t=t_1+t_2$. \\

\vspace{0.5cm}
\begin{tikzpicture}[scale=0.6]\label{Picofcocycle}
\draw [-] (0,0) -- (10,8);  
\path [->][red,bend right,dashed] (0,0) edge (3.75,3);
\path [->][red,bend right,dashed] (3.75,3) edge (10,8);
\path [->][blue,bend left,dashed] (0,0) edge (10,8);
\node  at (-.4,0)  {$0$};
\node  at (11,8)  {$t_1+t_2$};
\node  at (4.1,2.5)  {$t_2$};
\node  at (5.3, .9) {$\T_\circledast(t_2;L)$};
\node  at (11.3, 5) {$\T_\circledast(t_1;t_2\cdot L) $};
\node  at (1, 6) {$\T_\circledast(t_1+t_2;L)$};
\draw (4.2,-1.2) node[anchor=north] {\qquad\qquad \textbf{Figure A.} \textrm{Cocycle Property}};
\end{tikzpicture}\\

In fact, if we assume the time evolution is smooth, \eqref{cocycle} is equivalent to the property that $\T_\circledast$ satisfies the following first-order autonomous equation in the time variable $t$. 
\begin{Proposition}\label{Propf T'=BT}
For a cocycle $\T_\circledast$ we have the differential equation 
\begin{equation}\label{T'=BT}
\dfrac{d}{dt} \T_\circledast(t;L)=\B( t\cdot L) \T_\circledast(t;L),
\end{equation}
where $\B$ is a map from the set of Schr\"odinger operators \eqref{Schrodinger}  to the space of $2\times 2$ complex matrices. 

Conversely, if $\T_\circledast$ obeys the equation \eqref{T'=BT} for some trace zero matrix $\B$ dependent on $ t\cdot L$, then $\T_\circledast$ is a cocycle. 
\end{Proposition}

\begin{proof}
By differentiating \eqref{cocycle} with respect to $t_1$ and treating $t_2$ as a constant, we have that 
\begin{eqnarray*}
\dfrac{d}{dt_1}\T_\circledast(t_1+t_2;L)&=&\left(\dfrac{d}{dt_1}\T_\circledast(t_1; t_2\cdot L)\right)\T_\circledast(t_2;L),\\
\end{eqnarray*}
and
\begin{eqnarray*}
&&\left(\dfrac{d}{dt_1}\T_\circledast(t_1+t_2;L)\right)\T_\circledast^{-1}(t_1+t_2;L)\\
&=&\left(\dfrac{d}{dt_1}\T_\circledast(t_1;t_2\cdot L))\right)\T^{-1}_\circledast(t_1; t_2\cdot L).\\
\end{eqnarray*}
Now setting $t_1=0$  for this last equality  enables us to put 
\begin{eqnarray*}
\B( t_2\cdot L):&=&\left(\dfrac{d}{dt_1}\T_\circledast(t_1+t_2;L)\bigg\vert_{t_1=0}\right)\T^{-1}_\circledast(t_2;L)\\
&=&\left(\dfrac{d}{dt_1}\T_\circledast(t_1; t_2\cdot L)\bigg\vert_{t_1=0}\right)\T^{-1}_\circledast(0; t_2\cdot L)\\
&=& \left(\lim_{\epsilon\to 0}\frac{\T_\circledast(\epsilon; t_2\cdot L)-\T_\circledast(0; t_2\cdot L)}{\epsilon}\right)\T^{-1}_\circledast(0; t_2\cdot L).
\end{eqnarray*}
The limit exists since we assumed that $\T_\circledast$ represents a smooth evolution in time. We rename $t_2$ as $t$ and indeed we have found a function $\B$ that depends on $t\cdot L$ and satisfies \eqref{T'=BT}  

To prove the converse, we again set $t_2$ as a constant, and notice that the left and right sides of \eqref{cocycle} both solve the same differential equation $df/dt_1=\B((t_1+t_2)\cdot L) f$, and that they share the same initial conditions at $t_1=0$. Then the uniqueness of a solution of such a differential equation implies \eqref{cocycle}. 
\end{proof}
Note that in Proposition \ref{Propf T'=BT} $\textrm{trace }\B=0$, due to the fact that $\det \T_\circledast=1$. This is an elementary result, but we write down the proof below for the reader's convenience.

\begin{Proposition}\label{prop.traceB}
For any $t$, let $\B$ and $\T(t)$ be $2\times 2$ complex matrices. Assume that $\det \T(t)=1$ for all $t$, and that for some $\hat t$ the equation $\T'(\hat t)=\B\T(\hat t)$ holds. Then $\mathrm{trace}(\B)=0$. 
\end{Proposition} 
\begin{proof}
We write 
\begin{equation}
\T(\hat t)=\begin{pmatrix}
T_{11} & T_{12}\\
T_{21} & T_{22}
\end{pmatrix}, B=\begin{pmatrix}
B_{11} & B_{12}\\
B_{21} & B_{22}
\end{pmatrix}
\end{equation}
 If we look at the four entries of the matrices, $\T'(\hat t)=\B\T(\hat t)$ gives the system of equations

\begin{align}
T_{11}'=&B_{11}T_{11}+B_{12}T_{21}\label{e.BTNW}\\
T_{12}'=&B_{11}T_{12}+B_{12}T_{22}\label{e.BTNE}\\
T_{21}'=&B_{21}T_{11}+B_{22}T_{21}\label{e.BTSW}\\
T_{22}'=&B_{21}T_{12}+B_{22}T_{22}\label{e.BTSE}
\end{align} 
 We note that $T_{22}\times \eqref{e.BTNW}-T_{21}\times\eqref{e.BTNE}$ implies that $T_{11}'T_{22}-T_{12}'T_{21}=B_{11}$, and
$T_{11}\times \eqref{e.BTSE}-T_{12}\times\eqref{e.BTSW}$ implies that $T_{11}T_{22}'-T_{12}T_{21}'=B_{22}$. Adding those two equations obtains 
 \begin{equation}
 B_{11}+B_{22}=(T_{11}T_{22}-T_{21}T_{12})'=\dfrac{d}{dt}\det(\T(t))=0.
 \end{equation}

\end{proof}
Where $y$ is any solution to the eigenvalue equation $L(x)y=zy$, we observe that

\begin{equation*}
\dfrac{d}{dx} \begin{pmatrix} y \\ y_x \end{pmatrix}=
\begin{pmatrix} y_x \\ y_{xx} \end{pmatrix}=
\begin{pmatrix} 
0 & 1 \\ V(x,0)-z & 0
\end{pmatrix}
  \begin{pmatrix} y \\ y_x \end{pmatrix},
\end{equation*} 
  
 Thus by differentiating $\T_\boxplus$ with respect to $x$ and using \eqref{tmforschrodinger} and the Schr\"odinger equations $Lu=zu$, $Lv=zv$, we can show that 
\begin{equation}\label{T'=MT}
\dfrac{d}{dx} \T_\boxplus(x,z;L) =
\begin{pmatrix} 
0 & 1 \\ V(x,0)-z & 0
\end{pmatrix}
 \T_\boxplus(x,z;L)  \quad (=: \M\T_\boxplus). 
\end{equation}
Let us denote the first factor matrix on the right-hand side by $\M(x,z;L)$.

We have explained the cocycle property for evolutions in time, and we also derived a cocycle property for shifts in space. In fact it is true that the cocycle property can be extended to hold for evolutions in time and shifts in space jointly. This will be explained in more detail and in a more general setting in Subsection \ref{subs:group}.\\

 To construct the zero-curvature equation, we additionally assume that our time evolution  commutes with the shift in $x$, in the sense that shifting in $x$ first and then evolving by $t$ have the same effect on the flows as the reverse order (shifting in $x$ \emph{after} evolving by $t$). 
 Here we denote the effect of shift in $x$ by the same group action. In other words, since we assume this commutativity, we may thus write  $(x,t)\cdot$ instead of $t\cdot L(x)$, so now the group action is extended to the larger group $\R \times \R$. 
 
 Again, for notational convenience let us ignore the $z$-dependence of $\T_\boxplus$ and $\T_\circledast$. See the figure below.

\vspace{0.5cm}
\begin{tikzpicture}[scale=0.7]\label{Picofcommutativity}
\draw [red,thick,->] (0,0) -- (10,0);
\draw [red,thick,->] (10,0) -- (10,6);
\draw [blue,thick,->] (0,0) -- (0,6);
\draw [blue,thick,->] (0,6) -- (10,6);
\node at (-1,0) {$(0,0)$};
\node at (11.2,0) {$(x,0)$};
\node at (-1,6) {$(0,t)$};
\node at (11.2,6) {$(x,t)$};
\node [red] at (5,-.7) {$\T_\boxplus(x;L)$};
\node [red] at (13,3) {$\T_\circledast(t;(x,0)\cdot L)$};
\node [blue] at (-2,3) {$\T_\circledast(t;L)$};
\node [blue] at (5,6.7) {$\T_\boxplus(x;(0,t) \cdot L)$};
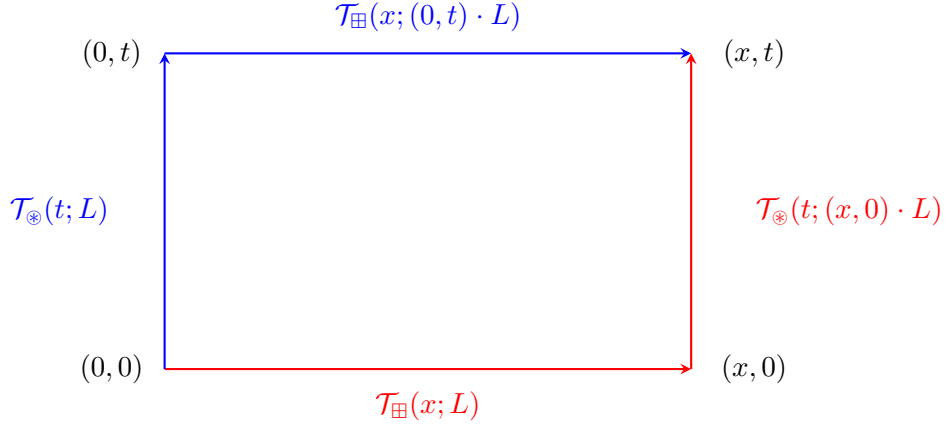
\draw (4.2,-2) node[anchor=north] {\qquad\qquad \textbf{Figure B.} \textrm{Commutativity of the $x$-shift and $t$-shift}};
\end{tikzpicture}\\
Since the two (blue and red) paths on the figure should have the same effect, 
\begin{equation}\label{comm with x}
\T_\circledast(t;(x,0)\cdot L)\T_\boxplus(x;L)=\T_\boxplus(x;(0,t) \cdot L)\T_\circledast(t;L).
\end{equation} 
We will show that this commutativity \eqref{comm with x} implies the zero-curvature equation
\begin{align*}
&\left.\dfrac{d}{dt} \M((0,t)\cdot L)\right\vert_{t=0}-\left.\dfrac{d}{dx} \B((x,0)\cdot L)\right\vert_{x=0} \\
=& -\M(L)\B(L)   + \B(L)\M(L),
\end{align*}
or simply
\begin{equation}\label{zceqn}
\partial_t \M-\partial_x \B=-[\M, \B]=-(\M\B-\B\M).
\end{equation} 
Indeed, by differentiating \eqref{comm with x} in $x$ and $t$, we see that 
\begin{eqnarray*}
&&\partial_t\partial_x \T_\circledast (t;(x,0)\cdot L) \T_\boxplus(x;L)+ \partial_t \T_\circledast (t;(x,0)\cdot L) \partial_x \T_\boxplus (x;L)\\
&& = \partial_t \partial_x \T_\boxplus(x;(0,t) \cdot L)\T_\circledast(t;L)+ \partial_x \T_\boxplus(x;(0,t) \cdot L)\partial_t \T_\circledast(t;L).
\end{eqnarray*}

After exchanging the order of derivatives of the first term on the left-hand side (which is fine due to the smoothness condition on $V$ and therefore on solutions to the related Schr\"odinger eigenvalue equation), we obtain
\begin{eqnarray*}
&&\partial_x\partial_t \T_\circledast (t;(x,0)\cdot L) \T_\boxplus(x;L)+ \partial_t \T_\circledast (t;(x,0)\cdot L) \partial_x \T_\boxplus (x;L)\\
&& = \partial_t \partial_x \T_\boxplus(x;(0,t) \cdot L)\T_\circledast(t;L)+ \partial_x \T_\boxplus (x;(0,t) \cdot L)\partial_t \T_\circledast(t;L).
\end{eqnarray*}
We then apply \eqref{T'=BT} and \eqref{T'=MT} to get
\begin{eqnarray*}
&\partial_x \B((0,t)\cdot (x,0)\cdot L) \,\, \T_\circledast(t;(x,0)\cdot L) \,\, T_\boxplus(x;L) \\
&+\B((0,t)\cdot (x,0)\cdot L)  \,\,  \partial_x\T_\circledast(t;(x,0)\cdot L)  \,\,  T_\boxplus(x;L)\\
&+\B((0,t)\cdot(x,0)\cdot L) \,\,  \T_\circledast(t;(x,0)\cdot L)  \,\,  \M(x;L)  \,\,  \T_\boxplus(x;L)\\ 
=& \partial_t\M(x;(0,t)\cdot L) \,\, \T_\boxplus(x;(0,t)\cdot L)  \,\,  \T_\circledast (t;L)\\
&+\M(x;t\cdot L) \,\, \partial_t\T_\boxplus(x;(0,t)\cdot L) \,\, \T_\circledast(t;L)\\
&+\M(x;(0,t)\cdot L) \,\, \T_\boxplus(x;(0,t)\cdot L) \,\, \B((0,t)\cdot L) \,\, \T_\circledast(t;L).
\end{eqnarray*}
 Set $x=t=0$, and note that $\T_\circledast (0,z;L)=\T_\boxplus (0,z;L)=\mathcal I_{2\times 2}$ (the identity matrix) for any $z$ and $L$ and 

$$
\partial_x\T_\circledast (t;(x,0)\cdot L)\Bigr\rvert_{x=t=0}=\partial_t\T_\boxplus(x,(0,t)\cdot L)\Bigr\rvert_{x=t=0}=
\left(\begin{smallmatrix} 0&0\\0&0 \end{smallmatrix}\right).
$$

We then obtain \eqref{zceqn}, as desired. \\
Lastly, we describe the polynomial recursion formalism on $\B(L)$ (to simplify notation, we will just write $\B$). For this, put 
$\B=\left( \begin{smallmatrix} A& C \\ -D & -A \end{smallmatrix} \right)$
 (due to $\textrm{trace }\B=0$)
and then insert this and \eqref{T'=MT} into the zero-curvature equation \eqref{zceqn}. 
By comparing entries, the following three equations should be satisfied: 
\begin{eqnarray}\label{threeeqns}
2A+C_x &=& 0  \nonumber  \\
(V-z)C+D+A_x &=& 0 \\
2(V-z)A+D_x+V_t &=& 0. \nonumber
\end{eqnarray}
In order to see the KdV hierarchy, assume that the entries $A$, $C$ and $D$ are polynomials in $z$. Due to the three equations \eqref{threeeqns} above, we can recursively construct the polynomial entries up to integral constants. 

More precisely, plugging the first two equations into the third and expressing it in terms of $C$ show that 
\begin{equation}\label{KdV hierarchy}
V_t=-\frac{1}{2} C_{xxx}+2(V-z)C_x+V_xC. 
\end{equation}
 We assume that the polynomials $C$ in $z$ (therefore, $A$ and $D$) are differential polynomials in $V$, i.e., these are polynomials of $V$ and their derivatives in $x$ ($V_x$, $V_{xx}$ and so on). Then we can derive the KdV hierarchy. We will demonstrate two easy examples.

If we assume $A,C,D$ are constants in $z$ (that is, degree zero polynomials in $z$), plugging that into \eqref{KdV hierarchy} gets us $V_t= C V_x$. In other words, evolution in time is equivalent to shifts in $x$. 

Let us consider instead the case where $A,C,D$ are polynomials in $z$ of degree at most $1$. We then consider \eqref{KdV hierarchy} with $C:=C_1z+C_0$, where $C_1$ and $C_0$ are constants in $z$ (but not necessarily in $x$). Plugging this in and then considering the $z^2$ term of \eqref{KdV hierarchy} we get that $0=-C_{1,x}$. This means that $C_1$ is a constant in $x$ as well as $z$. Now we consider the $z$-coefficients of \eqref{KdV hierarchy}. We obtain $0=-2C_{0,x}+V_xC_1$, which implies that $C_0=\frac{1}{2}VC_1+C_\star$, for some $C_\star$ constant in $x$ and $z$. Now when we consider the constant coefficients of \eqref{KdV hierarchy}, we obtain a time-evolution in $V$ that is a linear combination of the classical KdV equation \eqref{KdV equation} and the shift $V_t=V_x$. In fact, if we set $C_\star=0$ and $C_1=1$ we get \eqref{KdV equation} exactly.
 
 We can similarly obtain the higher-order members of the KdV hierarchy. We start by assuming that $C$ is a $n$-degree polynomial in $z$ for some $n\geq 2$, and then recursively apply
\begin{equation}\label{recursion}
C_{xxx}-4(V-z)C_x-2V_xC=0.
\end{equation}
We clarify what we mean with ``recursively":  we construct any (homogeneous) polynomials of degree at most $n-1$ satisfying \eqref{recursion} and some polynomial of the largest degree (which is $n$) should satisfy \eqref{KdV hierarchy}. Any linear combinations of these two polynomials then become those of the (infinite) family of KdV hierarchy. See Section 1 of \cite{Gesztesy-Holden} for more details. Thus \eqref{KdV hierarchy} gives the whole family of the KdV hierarchy when we vary our choice of the polynomial $C$.

\section{Canonical system flows}\label{Section3}
We now would like to apply a similar construction in Section \ref{Section:KdVreview} to canonical system flows. 
\subsection{Canonical systems}
Canonical systems are the equations  
\begin{equation}\label{cs systems}
\J\vec{u}_x(x,z;\HH)=z\HH (x)\vec{u}(x,z;\HH), \quad x\in\R,
\end{equation} 
where $\J=\big( \begin{smallmatrix}  0 & -1 \\ 1 &0 \end{smallmatrix}\big)$ and $\HH$ (known as a Hamiltonian) is a symmetric positive semi-definite $2\times 2$ matrix whose entries are real-valued and locally integrable functions. Here $z$ is a spectral parameter in $\C$, as in the Schr\"odinger eigenvalue equation. In particular, canonical systems are called trace-normed if $\textrm{Tr } H(x)=1$ for almost all $x$ in $\R$. 

These canonical systems are in some sense the highest generalization of many self-adjoint spectral problems. For instance, eigenvalue equations involving Jacobi matrices, Schr\"odinger operators (more generally, Sturm-Liouville operators) and Dirac operators can be written as canonical systems. See Proposition 8 in \cite{RemdeB} or Proposition 4.1 in \cite{Hur-density} to convert Schr\"odinger eigenvalue equations to unique trace-normed canonical systems and vice versa. For discrete cases such as Jacobi operators, see \cite{Hur-nondensity}.

In the viewpoint of inverse spectral theory, this generalization is important for three reasons. First, rewriting a Jacobi, Schr\"odinger, Dirac, etc. equation as a canonical system does not change its spectrum. Second,  there is one-to-one correspondence between Herglotz functions and trace-normed half-line canonical systems \cite{deB,RemdeB,Win}. Lastly, the set of all canonical systems becomes a compact topological space with the local uniform convergence of their related Herglotz functions (or Hamiltonians $\HH$ in \eqref{cs systems}) \cite{deB,Hur-density,K&L,RemdeB}. 

In the next subsection  we will try to find isospectral canonical system flows which look like KdV flows by the zero-curvature equation for canonical system flows (which is \eqref{zceqnforcs}). We will encounter obstructions to constructing these flows that are not present for KdV flows.

Note that unlike the Schr\"odinger equation, there is in general no operator associated with a canonical system. Instead of an operator we have a self-adjoint relation \cite {HSW}. This partly explains why we favor the zero-curvature equation approach rather than attempting to apply a Lax pair formalism via operators on canonical system flows. 

\subsection{Canonical system flows and their zero-curvature equation}
To consider a smooth  evolution of canonical systems \eqref{cs systems}, let us impose a $t$-dependence on $\HH$ in \eqref{cs systems}. We rewrite $\HH(x,t)$ as $(x, t)\cdot \HH$, where we interpret $x,t$ as shifts in space and time respectively.  Then the $t$-evolution on $\HH$ is what we call a canonical system flow. Denote $\HH$ and its determinant respectively by 
\begin{equation}\label{eq:HHdefn}
(x,t)\cdot \HH=\begin{pmatrix} f(x,t) & g(x,t) \\ g(x,t) & h(x,t) \end{pmatrix}
\end{equation}
and
\begin{equation}\label{Delta}
\Delta(x,t) :=\det ((x,t)\cdot \HH)=f(x,t)h(x,t)-(g(x,t))^2.
\end{equation}

Due to the positive semi-definiteness condition on $\mathcal{H}$, the functions $f$, $h$ and $\Delta$ are non-negative for all $t$ and $x$. 

Similar to the smooth condition on the potentials $V$ for KdV flows, assume the following: 
\begin{Hypothesis}\label{Hypo}
The entries $f$, $g$ and $h$ are smooth functions (jointly) on both $x$ and $t$. 
\end{Hypothesis}
As in the construction of the KdV hierarchy, let us impose  three conditions on evolution in $t$: a cocycle property on transfer matrices, commutativity with the shift in $x$ and a polynomial recursion formalism. We will use $\textrm{ }\widetilde{}\textrm{ }$  to denote objects that are associated to canonical system flows rather than KdV flows. 
For convenience we may suppress the $x$-, $t$-, $z$- or $\HH$-dependence when it is clear.

Similar to \eqref{tmforschrodinger}, let us first introduce transfer matrices $\widetilde{\T_\boxplus}$ (equivalent to $\rho$-matrix functions in \cite{K&L}) by 
\begin{equation*}
\widetilde{\T}_\boxplus(x,z; \HH)=\begin{pmatrix} u_1(x,z; \HH) & v_1(x,z; \HH) \\ u_2(x,z; \HH) & v_2(x,z; \HH) \end{pmatrix},
\end{equation*}
where two columns $\vec{u}=\left( \begin{smallmatrix} u_1 \\ u_2 \end{smallmatrix} \right)$ and $\vec{v}=\left( \begin{smallmatrix} v_1 \\ v_2 \end{smallmatrix} \right)$ are the solutions to \eqref{cs systems}. Note that we impose the following boundary conditions:  
\begin{equation}\label{bcat0forcs}
\begin{pmatrix}
u_1(0,z; \HH) & u_2(0,z; \HH) \\ v_1(0,z; \HH) & v_2(0,z; \HH)
\end{pmatrix}
=\begin{pmatrix} 1 & 0 \\ 0&1 \end{pmatrix}. 
\end{equation}
The matrix $\widetilde{\T}_\boxplus(x,z,\HH)$ transfers the solutions from $0$ to $x$ as well. That is, for any solution $u$ of \eqref{cs systems},

\begin{equation}\label{cs.xevolution}
\widetilde{\T}_\boxplus(x,z; \HH)
\begin{pmatrix}
u_1(0,z; \HH)\\
u_2(0,z; \HH)
\end{pmatrix}
=\begin{pmatrix}
u_1(x,z; \HH)\\
u_2(x,z; \HH)
\end{pmatrix}
=\begin{pmatrix}
u_1(0,z; (x,0)\cdot\HH)\\
u_2(0,z; (x,0)\cdot \HH)
\end{pmatrix}.
\end{equation}
Particularly, this applies when the $u$ represents the $\ell^2$ solutions of the canonical system.
 Since the Wronskian of two solution vectors to \eqref{cs systems} is constant, \eqref{bcat0forcs} implies that $\det \widetilde{\T}_\boxplus=1$ for all $x$ and $\HH$. 

We can now focus on the time evolution. Let us consider the $\mathrm{SL}_2(\mathbb C)$ transfer matrices $\widetilde{\T}_\circledast(t,\HH)$. Define $ u_{1,\pm}, u_{2,\pm}$ as the solution to \eqref{cs systems} that is $\ell^2$ at $\pm\infty$.  Then we want $\widetilde{\T}_\circledast(t,\HH)$ to satisfy

\begin{equation}\label{cs.tevolution}
\widetilde{\T}_\circledast(t,\HH)
\begin{pmatrix}
u_{1,+}(0,z; \HH)& u_{1,-}(0,z; \HH) \\ u_{2,+}(0,z; \HH) & u_{2,-}(0,z; \HH)
\end{pmatrix}
=
\begin{pmatrix}
u_{1,+}(t,z; \HH)& u_{1,-}(t,z; \HH) \\ u_{2,+}(t,z; \HH) & u_{2,-}(t,z; \HH)
\end{pmatrix}
\end{equation}
Again, in order for this time-evolution to make sense the $\widetilde{\T}_\circledast$ must obey the same cocycle property given in Definition \ref{Defn:cocycle}. We then observe that the proof of Proposition \ref{Propf T'=BT} may be copied verbatim to the canonical systems setting. Thus where $\B$ is a map from the set of canonical systems to the space of $2\times 2$ complex matrices, $T_\circledast$ is the unique solution of a differential equation of the form 
\begin{equation}\label{T'=BT for cs}
\dfrac{d}{dt} \widetilde{\T}_\circledast(t;\HH)=\widetilde{\B}((0,t)\cdot\HH) \widetilde{\T}_\circledast(t;\HH),
\end{equation}
with initial conditions
\begin{equation}
\widetilde{\T}_\circledast(0;\HH)=\begin{pmatrix}1&0\\ 0&1 \end{pmatrix}.
\end{equation}
Since $\widetilde{\T}_\circledast$ has determinant $1$, $\widetilde{\B}$ has trace zero by Proposition \ref{prop.traceB}.

Again, we may consider the action of both $\widetilde{\T}_\circledast$ and $\widetilde{\T}_\boxplus$ as a linear fractional transformation \eqref{e.LFT} on $\mathbb C\cup \{\infty\}$, and we define the $m$-functions 
\begin{equation}
m_{\pm}(z,\HH):=\pm \frac{u_{1,\pm}(0,z,\HH)}{u_{2,\pm}(0,z,\HH)}.
\end{equation}
For a more extensive discussion of $m$-functions corresponding to canonical systems, please consult Section 3.4 of \cite{RemBook}.

Given this perspective, we then observe that

\begin{align}
\widetilde{\T}_\boxplus (x;\HH)\odot m_{\pm}(z,\HH)=m_{\pm}(z,(x,0)\cdot \HH),\\
\widetilde{\T}_\circledast (t;\HH)\odot m_{\pm}(z,\HH)=m_{\pm}(z,(0,t)\cdot \HH).
\end{align}
Thus we may view $\widetilde{\T}_\boxplus$ and $\widetilde{\T}_\circledast$ as capturing evolutions of $m$-functions. 
We should clarify some potential confusion by summarizing the reasoning process. We first reason that any matrix $\widetilde{\T}_\circledast$ that captures a time-evolution of the form \eqref{cs.tevolution} must obey the cocycle property in Definition \ref{Defn:cocycle}. This implies that $\widetilde{\T}_\circledast$ must obey \eqref{T'=BT for cs} for some choice of $\B$. Then, given a $\B$, we define $\widetilde{\T}_\circledast$ as the unique solution of \eqref{T'=BT for cs}, and thus given that choice of $\B$ we can define the time-evolution of $u_{1,\pm}, u_{2,\pm}$ from \eqref{cs.tevolution}.

Let us now attempt to relate $\widetilde{\T}_\circledast$ with $\widetilde{\T}_\boxplus$.  With respect to $\widetilde{\T}_\boxplus$, \eqref{cs systems} and the fact that $\J^{-1}=-\J$ lead to the differential equations for $\widetilde{\T}_\boxplus$ in $x$
\begin{equation}\label{T'=MT for cs}
\dfrac{d}{dx} \widetilde{\T}_\boxplus(x,z)=-z\J ((x,0)\cdot\HH) \widetilde{\T}_\boxplus(x,z)\quad (=: \widetilde{M}\widetilde{\T}_\boxplus).
\end{equation}
Similar to \eqref{T'=MT} for KdV flows, define 
\begin{equation*}
\widetilde{M}:=-z\J \HH=z \begin{pmatrix} g & h \\ -f & -g \end{pmatrix}. 
\end{equation*}
By assuming that our evolution in $t$ commutes with the shift in $x$ and using \eqref{T'=BT for cs} and \eqref{T'=MT for cs}, we can construct the zero-curvature equation for \eqref{cs systems}. Both the zero-curvature equation and its proof are identical to that for \eqref{zceqn} except for the tildes. We write
\begin{equation}\label{zceqnforcs}
\dfrac{d}{dt} \widetilde{\M}-\dfrac{d}{dx} \widetilde{\B}=\widetilde{\B}\widetilde{\M}-\widetilde{\M}\widetilde{\B}.
\end{equation}

As the last condition, assume that $\widetilde{\B}$ have polynomial entries (in $z$), and denote them by
\begin{equation*}
\widetilde{\B}=\begin{pmatrix} \widetilde{A} & \widetilde{C} \\ -\widetilde{D} & -\widetilde{A} \end{pmatrix}.
\end{equation*}
Plugging the matrix above and $\widetilde{\M}$ into the zero-curvature equation \eqref{zceqnforcs} then leads to three conditions for entries of $\widetilde{\B}$ and $\HH$:
\begin{eqnarray}\label{threeeqnsforcs}
  zh_t-\widetilde{C}_x &=&  2zh \widetilde{A}-2zg \widetilde{C}  \nonumber   \\ 
-zf_t +\widetilde{D}_x &=& 2zf\widetilde{A}-2zg \widetilde{D}  \\
zg_t-\widetilde{A}_x &=& -zf \widetilde{C}+zh \widetilde{D}. \nonumber
\end{eqnarray}
In summary, these conditions \eqref{threeeqnsforcs} are equivalent to the zero-curvature equation \eqref{zceqnforcs} for the canonical system flows \eqref{cs systems}.

\subsection{Canonical system flows as a $\mathbb R\times\mathbb R$ group action}\label{subs:group}
We really think of our integrable flow as a group action on $\mathbb R\times\mathbb R$, where the first component represents the shift on the $x$-axis, and the second component represents the shift on time $t$.

It is clear from \eqref{cs.tevolution} that $\widetilde \T_\circledast$ obeys the $t$-cocycle property expressed in Definition \ref{Defn:cocycle}  . We can in a similar way deduce that by \eqref{cs.xevolution} $\widetilde \T_\boxplus$ also obeys the cocycle property in terms of shifts on $x$. In other words, for any $x,y\in\mathbb R$
\begin{equation}
\widetilde \T_\boxplus(x+y,\HH)=\widetilde \T_\boxplus(x,(y,0)\cdot\HH)\T_\boxplus(y,\HH).
\end{equation} 

Let us combine the maps $\widetilde\T_\circledast, \widetilde\T_\boxplus$ into a cocycle for a $\mathbb R\times\mathbb R$-group action on $\HH$.  We define
\begin{equation}\label{defn:jointgroup}
\widetilde \T((x,t);\HH):= \widetilde \T_\circledast(t;(x,0)\cdot \HH)\widetilde \T_\boxplus(x; \HH).
\end{equation}

We now have to prove that $\widetilde T$ obeys the joint cocycle property,
\begin{equation}\label{eq:jointcocycle}
\widetilde \T(g+h;\HH)= \widetilde \T(g;h\cdot \HH)\widetilde \T(h; \HH), \text{ for }g,h\in\mathbb R\times \mathbb R.
\end{equation}

\begin{Proposition}
The zero-curvature equation \eqref{zceqnforcs} and the joint cocycle condition \eqref{eq:jointcocycle} are equivalent.
\end{Proposition}
\begin{proof}
Showing that the joint cocycle condition implies the zero-curvature equation is straightforward. Obviously the existence of the joint cocycle on $x$ and $t$ implies the cocycle property for $x$ and $t$ individually; and this in turn implies \eqref{T'=BT for cs} and \eqref{T'=MT for cs} apply for $\widetilde\T$ (to clarify, we don't need the fact that $\M=-z\J(x,0)\cdot\HH$ from  \eqref{T'=MT for cs}, we just need that $\M$ is some $x$-dependent $2 \times 2$ complex matrix).  We can demonstrate this implication in the same way as the proof of Proposition \ref{Propf T'=BT}. Then \eqref{zceqnforcs} follows immediately.

The other direction is harder. First, we will demonstrate that it suffices to prove
\begin{equation}\label{eq:suffices}
\widetilde \T((0,t);(x,0)\cdot \HH)\widetilde \T((x,0);\HH) = \widetilde \T((x,0);(0,t)\cdot \HH)\widetilde\T((0,t);\HH).
\end{equation}
This is due to the commutativity of the group $\mathbb R\times \mathbb R$. Specifically, by replacing $x$ with $x_1$ and $\HH$ with $x_2\cdot\HH$ we have that 
\begin{align}\nonumber
&\widetilde \T((0,t);(x_1+x_2,0)\cdot \HH)\widetilde \T((x_1,0);(x_2,0)\cdot \HH) \\
&= \widetilde \T((x_1,0);(x_2,t)\cdot \HH)\widetilde\T((0,t);(x_2,0)\cdot \HH) \label{midstep}.
\end{align}
Note that \eqref{defn:jointgroup}) implies that the cocycle property applies for $x$ and $t$ individually (although we don't know this yet for $x$ and $t$ jointly). In other words for any $x$,$t$ and $\HH$,

\begin{align}
\widetilde \T((x,t);\HH)=& \widetilde \T((0,t);(x,0)\cdot \HH)\widetilde \T((x,0); \HH), \label{e.cocycletx}\\
\widetilde \T((x,t);\HH)=& \widetilde \T((x,0);(0,t)\cdot \HH)\widetilde \T((0,t); \HH). \label{e.cocyclext}
\end{align}

Thus for arbitrary $(x_1,t_1)$ and $(x_2,t_2)$ in $\mathbb R\times\mathbb R$,
\begin{align} 
&\widetilde \T((x_1+x_2,t_2+t_1);\HH)\nonumber\\
&=\widetilde \T((0,t_2+t_1);(x_1+x_2,0)\cdot\HH) \widetilde \T((x_1+x_2,0);\HH)\nonumber\\
&=\widetilde \T((0,t_1);(x_1+x_2,t_2)\cdot \HH)\widetilde \T((0,t_2);(x_1+x_2,0)\cdot \HH)\widetilde \T((x_1,0);(x_2,0)\cdot \HH) \widetilde \T((x_2,0); \HH)\nonumber\\
&=\widetilde \T((0,t_1);(x_1+x_2,t_2)\cdot \HH) \widetilde \T((x_1,0);(x_2,t_2)\cdot \HH)\widetilde \T((0,t_2);(x_2,0)\cdot \HH)\widetilde\T((x_2,0); \HH) \nonumber\\
&=\widetilde \T((x_1,t_1);(x_2,t_2)\cdot \HH) \widetilde\T((0,t_2);(x_2,0)\cdot \HH)\widetilde\T((x_2,0); \HH)\nonumber\\
&= \widetilde \T((x_1,t_1);(x_2,t_2)\cdot \HH) \widetilde\T((x_2,t_2); \HH)\nonumber.
\end{align}
Here we have used \eqref{midstep} to go from the third line to the fourth. The first equality and the last two equalities follow from \eqref{e.cocycletx}.

It remains to prove \eqref{eq:suffices}. For convenience the left and right sides of \eqref{eq:suffices} are denoted by $\ell$ and $r$ respectively. It is easy to check that their initial conditions at $x=t=0$ are the same. 

Note that by \eqref{T'=BT for cs} and \eqref{defn:jointgroup} we have
\begin{equation*}
\dfrac{\partial}{\partial t}\ell=\widetilde\B((x,t)\cdot\HH)\ell
\end{equation*}
and
\begin{eqnarray}\label{eq:diffr}
\dfrac{\partial}{\partial t}r &=& \left ( \dfrac{\partial}{\partial t} \widetilde\T((x,0);(0,t)\cdot \HH)\right)\widetilde\T((0,t);\HH) \\ \label{eq:diffr}
&&+ \,  \widetilde\T((x,0);(0,t)\cdot\HH)\widetilde\B((0,t)\cdot \HH)\widetilde\T((0,t);\HH). \nonumber
\end{eqnarray}

Our goal is to have $\ell$ and $r$ solve the same differential equation, so we are done if we can express the right hand side of \eqref{eq:diffr} in the form $ \widetilde\B((x,t)\cdot\HH)r$. This will obviously follow from
\begin{align*}
\nonumber&\left(\dfrac{\partial}{\partial t} \widetilde\T((x,0);(0,t)\cdot \HH)\right)\widetilde\T^{-1}((x,0);(0,t)\cdot \HH)\\
 \nonumber &+ \widetilde\T((x,0);(0,t)\cdot \HH) \widetilde\B((0,t)\cdot H)\widetilde\T^{-1}((x,0);(0,t)\cdot \HH)\\
  &= \widetilde\B((x,t)\cdot\HH).
\end{align*}
Replacing $(0,t)\cdot\HH$ with $\HH$, this is equivalent to
\begin{equation}
\dfrac{\partial}{\partial t} \widetilde\T((x,0);(0,t)\cdot\HH)|_{t=0} = \widetilde\B((x,0)\cdot \HH)\widetilde\T((x,0);\HH) - \widetilde\T((x,0);\HH)\widetilde\B(\HH). \label{eq:SecondStage}
\end{equation}
We perform the same trick again. We denote the left hand side of \eqref{eq:SecondStage} as $\tilde\ell$ and the right hand side as $\tilde r$. Observe that they have the same initial conditions at $x=0$. Note that, using \eqref{T'=MT for cs},
\begin{equation}
\dfrac{\partial}{\partial x}\tilde\ell =  \widetilde \M((x,0)\cdot\HH)\tilde\ell(x) + \dfrac{\partial}{\partial t}{\widetilde\M}((x,0)\cdot\HH)\widetilde\T((x,0);\HH)\label{eq:tildeell}
\end{equation} 
and 
\begin{align}
\dfrac{\partial}{\partial x}\tilde r =& \dfrac{\partial}{\partial x}\widetilde\B((x,0)\cdot\HH)\widetilde\T((x,0);\HH) \label{eq:tilder}\\
&+ \widetilde\B((x,0)\cdot \HH)\widetilde\M((x,0)\cdot\HH)\widetilde\T((x,0);\HH) \nonumber\\
&- \widetilde\M((x,0)\cdot\HH)\widetilde\T((x,0);\HH)\widetilde\B(\HH).     \nonumber
\end{align}
The zero-curvature equation \eqref{zceqnforcs} together with \eqref{eq:tilder} implies that \eqref{eq:tildeell} still holds if we replace $\tilde\ell$ with $\tilde r$. Thus $\tilde \ell=\tilde r$ and this concludes our proof.
\end{proof}

\section{Main result}
In this section we would like to see how difficult it is to find canonical system flows \eqref{cs systems} satisfying the three equations \eqref{threeeqnsforcs}. 

\begin{Theorem}\label{mainthm}
Let $f(x,t),g(x,t),h(x,t)$ be the entries of $(x,t)\cdot\HH$, as in \eqref{eq:HHdefn} and assume that they are smooth on $x$ and $t$ and that they are differential polynomials in $z$ of degree at least two. Then there is no canonical system flow \eqref{cs systems} satisfying the three equations \eqref{threeeqnsforcs}, unless $(\det ((x,t)\cdot \HH))^{-1/2}$ is a differential polynomial in $f(x,t)$, $g(x,t)$ and $h(x,t)$.
\end{Theorem}
\newpage
\noindent
\textbf{Remarks}
\begin{enumerate}
\item The condition that the entries of $\widetilde{\B}$ are differential polynomials of $f$, $g$ and $h$ is reasonable, since the same restriction applies for the corresponding $\B$ for KdV flows.  For KdV setting, the potentials $V$ are expressed as 
\begin{equation}\label{V in detH}
V=\frac14 \det \left((x,t)\cdot\HH^V_{xx}\right),
\end{equation} 
where $\HH^V$ are the Hamiltonians of the canonical systems which are re-written as the Schr\"odinger equations $t\cdot Ly=zy$. As discussed before, the entries of $\B$ are differential polynomials in $V$, and they are therefore differential polynomials of the entries of $\HH^{V}$. In other words, since $V$ is a differential polynomial of the entries of $\HH^{V}$ and the entries of $\B$ are differentia polynomials of $V$, the entries of $\B$ are differential polynomials of the entries of $\HH^{V}$.  This is exactly the same condition as the one in our main theorem. 
For \eqref{V in detH}, it turns out that 
\begin{equation*}
\HH^V=\begin{pmatrix} u^2_0 & u_0v_0 \\ u_0v_0 & v^2_0 \end{pmatrix}
\end{equation*}
where $u_0$ and $v_0$ are the solutions for zero energy, i.e., $Ly=0$ with $V$ satisfying \eqref{bcat0} with $u_0$ and $v_0$ instead. Due to the fact that $u_{0,xx}=Vu_0$ and $v_{0,xx}=Vv_0$, direct computation shows the expression 
\begin{equation*}
\det (\HH^V_{xx})=4V(u_0v_{0,x}-u_{0,x}v_0)^2=4V.
\end{equation*} 

\item It is notable that, even if we assume that $\Delta$ is constant (maybe this is the easiest case), the canonical system flows \eqref{cs systems} or three equations \eqref{threeeqnsforcs} would not be significantly easier. For example, the zero $\Delta$ can reduce  three equations \eqref{threeeqnsforcs} to two equations. However, it does not give any better information about the existence of the global solutions to \eqref{cs systems} satisfying these two conditions.

\item It is known that the Schr\"odinger eigenvalue problem can be rewritten as a canonical system. See, for instance Section 1.3 of \cite{RemBook}. The KdV flow can be thought of as a time-evolution of the Schr\"odinger eigenvalue problem. So why are the obstructions we discover absent for the KdV hierarchy? This is because the canonical system flow is not really a generalization of the KdV flow. The shift in $x$ on KdV flows is different from the one on canonical system flows. More precisely, when we express a Schr\"odinger equation as a canonical system, the $x$-shift in the Schr\"odinger equation does not correspond to the $x$-shift in the canonical system. The true generalization of the KdV flow is rather the \textit{twisted} canonical system flows studied in \cite{Remflow}.   

\item Our result shows that there are no canonical system flows corresponding to polynomials of degree two or more (unless perhaps $(\det ((x,t)\cdot\HH))^{-1/2}$ is a differential polynomial for every $t$, which is a rather restrictive condition). Flows of degree zero correspond to shifts in $x$.  Finding flows of degree one requires us to solve an existence problem involving a system of nonlinear PDEs. Figuring out if these degree one flows exist would be an interesting open problem.
\end{enumerate}

To prove Theorem \ref{mainthm}, the following lemma is easy but crucial, and it gives insight to the reasons for the conditions in Theorem \ref{mainthm}. 
\begin{Lemma}\label{condition on Delta}
For the three equations \eqref{threeeqnsforcs} to be consistent, the determinants $\Delta$ of Hamiltonians $\HH$ ,must satisfy the condition 
\begin{equation}\label{Deltacondition}
z\Delta_t= f \widetilde{C}_x + h \widetilde{D}_x - 2g \widetilde{A}_x.
\end{equation}
\end{Lemma}
The condition \eqref{Deltacondition} strongly suggests that it is difficult to have canonical system flows similar to KdV flows, since  the independence of $\Delta$ on $z$ tells that the left-hand side is (at most) linear in $z$.  However, the right-hand side consists of polynomials of any degree which can be chosen at our disposal.  As polynomials in $z$, the equality \eqref{Deltacondition} would thus be very difficult to satisfy. 
\begin{proof}
The idea is to express three equations \eqref{threeeqnsforcs} as a matrix equation: 
\begin{equation}\label{matrixeqnforcs}
z
\begin{pmatrix} 2h & -2g & 0 \\ 2f&0&-2g\\ 0&-f&h \end{pmatrix}
\begin{pmatrix} \widetilde{A} \\ \widetilde{C} \\ \widetilde{D}\end{pmatrix}
=\begin{pmatrix} zh_t-\widetilde{C}_x\\ -zf_t+\widetilde{D}_x \\ zg_t-\widetilde{A}_x\end{pmatrix}.
\end{equation}
Observe that  the coefficient matrix above has determinant $0$. By applying elementary row operations or multiplying some vector in the kernel of the coefficient matrix from the left, \eqref{matrixeqnforcs} reads  
\begin{equation*}
z
\begin{pmatrix} 2fh & -2fg & 0 \\ 0&0&0\\ 0&-fg&gh \end{pmatrix}
\begin{pmatrix} \widetilde{A} \\ \widetilde{C} \\ \widetilde{D}\end{pmatrix}
=\begin{pmatrix}zfh_t-f\widetilde{C}_x\\ -z\Delta_t+f \widetilde{C}_x + h \widetilde{D}_x - 2g \widetilde{A}_x \\ zgg_t-g\widetilde{A}_x \end{pmatrix}.
\end{equation*}
Thus \eqref{matrixeqnforcs} has a solution only if \eqref{Deltacondition} is satisfied. 
\end{proof}
\textbf{Remark.} A similar process can be applied to the KdV setting. 
A version of the three equations in \eqref{threeeqns} for KdV flows can be written as the matrix equation
\begin{equation*}
\begin{pmatrix} 2&0&0\\ 0& V-z &1 \\ 2(V-z) &0&0\end{pmatrix}
\begin{pmatrix}  A\\C\\D \end{pmatrix}
=
\begin{pmatrix}  -C_x \\ -A_x \\ -V_t-D_x \end{pmatrix}.
\end{equation*}
By the Gauss elimination method, we obtain another reduced matrix equation
\begin{equation*}
\begin{pmatrix} 2&0&0\\ 0& V-z &1 \\ 0 &0&0 \end{pmatrix}
\begin{pmatrix}  A\\C\\D \end{pmatrix}
=
\begin{pmatrix}  -C_x \\ -A_x \\ -V_t-D_x+(V-z)C_x \end{pmatrix},
\end{equation*}
which implies that 
\begin{equation*}
V_t=-D_x+(V-z)C_x.
\end{equation*}
This can be considered as a ``prototype'' of the KdV hierarchy. This is because first two equations on \eqref{threeeqns}, which are $A=-\frac12 C_x$ and $D=-A_x-(V-z)C$, imply that  
\begin{equation*}
V_t=-D_x+(V-z)C_x \implies V_t=-\frac12 C_{xxx}+2(V-z)C_x+V_xC,
\end{equation*}
where the last equation is the KdV hierarchy \eqref{KdV hierarchy}. All this means that \eqref{Deltacondition} can be thought of as a prototype of the hierarchy for canonical system flows \eqref{cs systems} satisfying three equations \eqref{threeeqnsforcs}.

\begin{proof}[Proof of Theorem \ref{mainthm}]
Assume the entries of $\widetilde{B}$ are polynomials in $z$ of degree at least two, and write 
$
\widetilde{A}=A_0+A_1z+\cdots+ A_n z^{n}
$
with similar notation for $\widetilde{C}$ and $\widetilde{D}$. 

Since the left-hand sides on \eqref{threeeqnsforcs} are of degree at most $n\ge 2$, comparing the coefficients for the greatest power $z^{n+1}$ shows that 
\begin{equation}\label{eqnsforz^n+1}
hA_n=gC_n, \quad fA_n=gD_n \quad \textrm{and} \quad fC_n=hD_n.
\end{equation}
Then we can introduce a $K$ such that 
$A_n=gK$, $C_n=hK$ and $D_n=fK$ solve \eqref{eqnsforz^n+1}. Note that $K$ must be a differential polynomial in $f,g,h$.
Observe that \eqref{Deltacondition} shows
\begin{equation*}
fC_{k,x}+hD_{k,x}-2g A_{k,x}=0 \quad \textrm{for}\quad  2\leq k \leq n.
\end{equation*}
By plugging \eqref{eqnsforz^n+1} on the condition above when $k=n$ $(\ge 2)$, we obtain the following equation for $K$ and $\Delta$:
\begin{equation}\label{eqnforK}
\Delta_x K+2\Delta K_x=0,
\end{equation}
where $\Delta=\det \HH=fh-g^2$ (which is \eqref{Delta}).
Solving the ODE \eqref{eqnforK} gives us that, for some constant $\kappa$,
\begin{equation*}
K=\frac{\kappa}{\sqrt{\Delta}}.
\end{equation*}
Our conditions on $K$ then imply that $\Delta^{-1/2}$ has to be a differential polynomial in $f$, $g$ and $h$.
\end{proof}

\end{document}